\documentclass[10pt,reqno]{amsart}

\usepackage[applemac]{inputenc}
\usepackage[T1]{fontenc}
\usepackage{loc_def-2}
\usepackage{amssymb}
\usepackage{microtype}
\usepackage{tikz}
\usepackage{calc}

\newtheorem{theorem}{Theorem}
\newtheorem{proposition}[theorem]{Proposition}
\newtheorem{lemma}[theorem]{Lemma}
\newtheorem{corollary}[theorem]{Corollary}

\theoremstyle{definition}

\newtheorem{example}[theorem]{Example}

\def\ScaleFigure{\ratio{\textwidth}{10cm}} 

\begin{document}

\title[Classifying Spaces with Virtually Cyclic
Stabilisers]{Classifying Spaces with Virtually Cyclic Stabilisers for
Certain Infinite Cyclic Extensions} 

\author{Martin Fluch}
\address{Martin Fluch, School of Mathematics, University of 
Southampton, Southampton, SO17 1BJ, United Kingdom}
\email{m.fluch@soton.ac.uk}

\date{July 6, 2010}

\begin{abstract}
    Let $G := B\rtimes \Z$ be an infinite cyclic extension of a group
    $B$ where the action of $\Z$ on the set of conjugacy classes of
    non-trivial elements of $B$ is free.  This class of groups
    includes certain strictly descending HNN-extensions with abelian
    or free base groups, certain wreath products by $\Z$ and the
    soluble Baumslag--Solitar groups $BS(1,m)$ with $|m|\geq 2$.  We
    construct a model for $\uu EG$, the classifying space of $G$ for
    the family of virtually cyclic subgroups of $G$, and give bounds
    for the minimal dimension of $\uu EG$.  Finally we determine the
    geometric dimension $\uugd G$ when $G$ is a soluble
    Baumslag--Solitar group.
\end{abstract}

\maketitle

\setlength{\baselineskip}{1.26\baselineskip}

\thispagestyle{empty}

%
%

\section{Introduction}

Let $G$ be a (discrete) group and let $\frakF$ be a \emph{family} of
subgroups of $G$.  That is $\frakF$ is non-empty and closed under
conjugation and taking subgroups.  A \emph{classifying space} of $G$
for the family $\frakF$ is a $G$-CW-complex $X$ such that
\begin{enumerate}
    \item the fixed point space $X^{H}$ is contractible if $H\in \frakF$;

    \item the fixed point space $X^{H}$ is empty if $H$ is any 
    subgroup of $G$ that does not belong to $\frakF$.
\end{enumerate}
We also say that $X$ is a model for $E_{\frakF}G$. The quotient space 
$X/G$ is called a model for $B_{\frakF}G$.

Classifying spaces exist for any family~$\frakF$ and are unique up
to equivariant homotopy~\cite{luck-89}.  However in order to do some
computation with classifying spaces it is important to know nice
representatives in the $G$-homotopy class of models for $E_{\frakF}G$.
Generally a model for a classifying space is considered nice if it
satisfies some finiteness conditions, for example being finite
dimensional or being of finite type (only finitely many equivariant
cells in each dimension).  For more details see for
example~\cite{luck-00a}.

In the case that $\frakF = \{1\}$ is the trivial family of subgroups,
the universal cover of an Eilenberg--Mac~Lane space $K(G,1)$ is a
model for $EG:=E_{\frakF}G$.  A classifying space for the family
$\frakF = \Ffin(G)$ of finite subgroups of~$G$ is also known as the
universal $G$-space for proper actions and denoted by
$\underline{E}G$.  For these families many nice models are known.
For more details we refer to the comprehensive survey article on
classifying spaces by Lück~\cite{luck-05}.

The object of study in our text is the classifying space for the
family $\frakF = \Fvc(G)$ of virtually cyclic subgroups which is
denoted by $\uu EG$.  Recently classifying spaces for this family of
subgroups caught the interest of the mathematical community due to its
appearance as the geometric object in the Farrell--Jones conjecture in
algebraic $K$- and $L$-theory (this conjecture is originally stated
in~\cite{farrell-93}).  In contrast to models for $EG$ and
$\underline{E}G$, the classifying spaces for virtually cyclic
subgroups are not well understood.  Classes of groups that are
understood are word hyperbolic groups~\cite{juan-pineda-06}, virtually
polycyclic groups~\cite{luck-07}, relative hyperbolic
groups~\cite{lafont-07} and
$\operatorname{CAT}(0)$-groups~\cite{luck-09, farley-09}.
Furthermore, there exist general constructions for finite index
extensions~\cite{luck-00a} and direct limits of groups~\cite{luck-07}.
Some more specific constructions can also be found
in~\cite{connolly-06} and~\cite{manion-08}.

Among the classes of groups for which there are no known nice models
for $\uu EG$ are soluble groups or, more broadly, elementary amenable
groups.  Even in the case that $G$ is metabelian there is no known
general constructions for a nice model for $\uu EG$.  In a recent
paper~\cite{kochloukova-10} it has been shown that one cannot expect
any finite type model for $\uu EG$ when $G$ is an elementary amenable
group.  In general it is conjectured that the there exists no finite
model for $\uu EG$ when $G$ is not virtually
cyclic~\cite{juan-pineda-06}.

It turns out that virtually cyclic extensions are the major source of
obstruction to a general construction of nice models for $\uu EG$, see
for example Theorem~5.1 in~\cite{martinez-perez-02}.  As mentioned
above, there exists a general construction for finite index
extensions.  However there is no known general construction for an
infinite cyclic extension, not to mention the case of a virtually
cyclic extension.

In this article we will consider infinite cyclic extensions of an
arbitrary group~$B$.  These are precisely the semidirect products $G =
B \rtimes \Z$.  We will show how to construct a model for $\uu EG$
when a model for $\uu EB$ is given, provided that the following
condition on the extension is satisfied: $\Z$ acts freely by
conjugation on the set of conjugacy classes of non-trivial elements of
$B$.  This construction will be such that a finite dimensional model
for $\uu EB$ will give a finite dimensional model for $\uu EG$.  The
resulting models will be far away from being of finite type.  One
major ingredient in our paper is an adaptation of a construction of
Juan-Pineda and Leary in~\cite{juan-pineda-06}.

The main result of this paper is Theorem~\ref{thrm:dimension}, where we
give bounds on the minimal dimension a model for $\uu EG$ can have.
The class of groups to which this theorem can be applied to includes
certain strictly descending HNN-extensions with abelian or free base
groups, certain wreath products by $\Z$ and the soluble
Baumslag--Solitar groups $BS(1,m)$ with $|m|\geq 2$.  Furthermore, in
the latter case we will give an explicite construction of a model for
$\uu EG$ of minimal dimension.  This brings us to our second main
result, Theorem~\ref{thrm:BSgroups}, which gives a precise answer for
the possible minimal dimensions of a model for $\uu EG$ when $G$ is a
soluble Baumslag--Solitar group.

This article is part of the author's Ph.D.~thesis, supervised by Brita
Nucinkis at the University of Southampton.  The author wishes to
express his gratitude to Brita Nucinkis, Ashot Minasyan, Armando
Martino and Holger Reich for their inspiring conversation and very
valuable comment which helped the author a lot in the work on this
article.  Thanks are also due to Andrew Warshall who kindly made the
author aware of a missing assumption in Lemma~\ref{lem:aux1}.

%
%

\section{Technical Preparations}

Let $B$ be a group and let $\varphi\in \Aut(B)$. 
Recall that a model for the semidirect product
\begin{equation*}
    G := B\rtimes \Z
\end{equation*}
where $\Z$ acts on $B$ via the automorphism $\varphi$ is the set
$B\times \Z$ with the multiplication given by
\begin{equation*}
    (x,r)\cdot (y,s) := (x\varphi^{r}(y),r+s).
\end{equation*}
The identity is $(1,0)$ and the inverse of any element $(x,r)$ is
given by $(x,r)^{-1}= (\varphi^{-r}(x^{-1}),-r)$. The group $B$ 
is embedded via $x\mapsto (x,0)$ as a normal subgroup of~$G$ and we 
consider $\Z$ embedded as a subgroup of $G$ via $r\mapsto (1,r)$.

\begin{lemma}
    \label{lem:aux1}
    Assume that $B$ is torsion free and does not contain a subgroup
    isomorphic to $\Z^{2}$.  Then $\Z$ acts freely by conjugation on
    the set of conjugacy classes of the non-trivial elements of $B$ if
    and only if $G$ does not contain a subgroup isomorphic to
    $\Z^{2}$.
\end{lemma}

\begin{proof}
    ``$\Leftarrow$'':\quad Suppose that $H$ is a subgroup of $G$ which
    is isomorphic to $\Z^{2}$.  Since $H$ cannot be contained in $B$
    it follows that there exists $(x,r)\in H\setminus B$.  On the
    other hand $H$ cannot have a non-trivial intersection with $B$ 
    and thus there exist a non-trivial $(y,0)\in B\cap H$. Then the 
    commutator
    \begin{equation*}
        [(x,r),(y,0)] = (x\varphi^{r}(y)x^{-1}y^{-1},0)
    \end{equation*}
    must be trivial which is the case if and only 
    if $x\varphi^{r}(y)x^{-1}y^{-1} = 1$. This implies
    that $\varphi^{r}(y)$ and $y$ belong to the same conjugacy class 
    in $B$. Since $r\neq 0$ and $y\neq 1$ this implies that 
    $\Z$ does not act freely on the set of conjugacy classes of 
    non-trivial elements of $B$. 
    
    \smallskip
    
    ``$\Rightarrow$'':\quad
    Suppose that $\Z$ does not 
    act freely on the set of conjugacy classes of non-trivial 
    elements of $B$. Then there exists $1\neq y\in B$ 
    and $0\neq r\in \Z$ such that $\varphi^{r}(y) = x^{-1}yx$ for 
    some $x\in B$. This implies that the non-trivial elements $(x,r)$ 
    and $(y,0)$ commute. In general $(x,r)$ has infinite order and 
    since $B$ is assumed to be torsion free it follows that the order 
    of $(y,0)$ is also infinite. Therefore $(x,r)$ and $(y,0)$ generate a 
    subgroup of $G$ which is isomorphic to $\Z^{2}$.
\end{proof}

The statement in the next lemma is only non-trivial if $B$ has
torsion.
 
\begin{lemma}
    \label{lem:aux2}
    Assume that $\Z$ acts freely via conjugation on the set of 
    conjugacy classes of non-trivial elements of $B$. If $H$ be a 
    virtually cyclic subgroup of $G$ which is not a subgroup of $B$ 
    then $H$ is infinite cyclic.
\end{lemma}

\begin{proof}
    First we note that all $x\in G\setminus B$ have infinite order. 
    Hence the subgroup $\tau(H)$ of $H$ which is generated by all 
    the elements of $H$ which have finite order is a subgroup of 
    $B$.
 
    By assumption there exists $(x,r)\in H$ with $r\neq 0$.  Then the
    infinite cyclic subgroup of $H$ generated by this element has
    trivial intersection with $B$ and therefore also has trivial
    intersection with $\tau(H)$.  Since $H$ is virtually cyclic this
    can happen only if $\tau(H)$ is finite.
     
    Assume towards a contradiction that there exists a non-trivial
    $(y,0)\in \tau(H)$.  Then $(z_{k},0) := (y,0)^{(x,r)^{k}}$, $k\in
    \N$, is a sequence of elements in $\tau(H)$ such that for each
    $k\in \N$ the element $z_{k}$ is conjugate in $B$ to
    $\varphi^{-rk}(y)$.  This claim is verified by induction.  The
    case $k=0$ is trivial.  Thus assume that $(z_{k},0)\in \tau(H)$
    and that there exists $u\in B$ such that $z_{k} =
    u^{-1}\varphi^{-rk}(y)u$.  Since $\tau(H)$ is a characteristic
    subgroup of $H$ and $(x,r)\in H$ it follows that $(z_{k+1},0) =
    (z_{k},0)^{(x,r)}\in \tau(H)$.  Furthermore we have that
    \begin{equation*}
	(z_{k+1},0) = (z_{k},0)^{(x,r)} = ( \varphi^{-r}(x^{-1}z_{k}x),0)
    \end{equation*}
    and hence $z_{k+1} = \varphi^{-r}(x^{-1}u^{-1}\varphi^{-rk}(y)ux)
    = v^{-1} \varphi^{-r(k+1)}(y) v$ with $v:= \varphi^{-r}(gx)$.
    That is that $z_{k+1}$ is conjugate in $B$ to
    $\varphi^{-r(k+1)}(y)$.  By assumption $\Z$ acts freely via
    conjugation on the conjugacy classes of non-trivial elements of
    $B$ and hence all $z_{k}$ belong to different conjugacy classes.
    In particular they are all pairwise different.  Thus $\{ (z_{k},0)
    : k\in \N\}$ forms an infinite subset of $\tau(H)$.  But this is a
    contradiction as we have shown above that $\tau(H)$ is finite!
    Therefore $\tau(H)$ must be trivial.  It follows that the
    virtually cyclic group $H$ does not have any torsion and thus it
    must be infinite cyclic.
\end{proof}

\begin{lemma}
    \label{lem:aux3}
    Assume that $\Z$ acts freely via conjugation on the set of 
    conjugacy classes of non-trivial elements of $B$. Then for any 
    $(x,r)\in G\setminus B$ and $y\in B$ we have
    \begin{equation*}
        (x,r)^{y} = (x,r) \iff y=1.
    \end{equation*}
\end{lemma}

\begin{proof}
    $(x,r)^{y} = (x,r)$ is equivalent to $\varphi(y) = xyx^{-1}$, 
    which is by assumption on the action of $\Z$ on $B$ equivalent to 
    $y=1$.
\end{proof}

\begin{lemma}
    \label{lem:aux4}
    Under the assumptions of the previous lemma, if $H$ is an infinite
    cyclic subgroup of $G$ that is not a subgroup of $B$, and $y\in
    B$, then $|H\cap H^{y}| = \infty$ if and only if $y=1$.
\end{lemma}

\begin{proof}
    The ``if'' statement is trivial.  Therefore assume that $y\neq 1$
    and let $(x,r)$ be a generator of $H$.  Then $r\neq 0$ and
    \begin{equation*}
	(z,r) := (x,r)^{y}\neq(x,r)
    \end{equation*}
    is a generator of $H^{y}$ where the inequality is due to
    Lemma~\ref{lem:aux3}.  Suppose, for a contradiction, that
    $|H\cap H^{y}| = \infty$.  Then there must exist $k,l\in
    \Z\setminus \{0\}$ such that $(x,r)^{k} = (z,r)^{l}$.  In
    particular this implies that $k=l$. But then we get 
    \begin{equation*}
        (z,r)^{l} = (z,r)^{k} =  \bigl((x,r)^{y}\bigr)^{k}  = 
	\bigl((x,r)^{k}\bigr)^{y} \neq (x,r)^{k},
    \end{equation*}
    where the last inequality is again due to Lemma~\ref{lem:aux3},
    and so we achieve our desired contradiction.  Hence we must have
    $|H\cap H^{y}| \neq \infty$.
\end{proof}

As in~\cite[p.~5]{luck-07} we define an equivalence relation
``$\sim$'' on the set $\Fvc(G)\setminus \Ffin(G)$ by
\begin{equation*}
    H\sim K :\!\!\iff |H\cap K| = \infty.
\end{equation*}
We denote by $[H]$ the equivalence class of the group $H$.  If $K$ is
not finite then $K\leq H$ implies that $K\sim H$.  Furthermore the
equivalence relation satisfies $H\sim K$ if and only $H^{g}\sim
K^{g}$.  Therefore the action of $G$ by conjugation on the set
$\Fvc(G)\setminus \Ffin(G)$ gives an action of $G$ on the set of
equivalence classes.  If $[H]$ is an equivalence class, then we denote
by $G_{[H]}$ the stabiliser of~$[H]$.

Given subgroup $H$ of $G$, the \emph{commensurator} $\Comm_{G}(H)$ of
$H$ in $G$ is defined as the subgroup
\begin{equation*}
    \Comm_{G}(H):=\{ g\in G :  |H : H\cap H^{g}| \text{ and } 
    |H^{g} : H\cap H^{g}| \text{ are finite} \}.
\end{equation*}
This subgroup is also known as the \emph{virtual normaliser}
$VN_{G}(H)$ of the subgroup $H$ in $G$.  In general it contains the
normaliser $N_{G}(H)$ of $H$ in $G$ as its subgroup.  In the case that
$H$ is a virtually cyclic subgroup of $G$ which is not finite we have
\begin{equation*}
    \Comm_{G}(H) = \{ g\in G : |H\cap H^{g}| = \infty\}.
\end{equation*}
In particular we have that $\Comm_{G}(H) = G_{[H]}$ in this case.

\begin{lemma}
    \label{lem:aux5}
    Assume that $\Z$ acts freely by conjugation on the set of
    non-trivial conjugacy classes of non-trivial elements of $B$.
    Then $\Comm_{G}(H)$ is infinite cyclic for any
    virtually cyclic subgroup $H$ of $G$ that is not a subgroup
    of~$B$.
\end{lemma}

\begin{proof}
    Any such virtually cyclic subgroup $H$ of $G$ is infinite cyclic
    by Lemma~\ref{lem:aux2}.  Therefore $G_{[H]} = \Comm_{G}(H)$.
    Suppose that $G_{[H]}$ is not infinite cyclic.  Then the canonical
    projection $\pi\: B\rtimes Z \to \Z$ cannot map $G_{[H]}$
    isomorphically onto its image.  Hence there exists a non-trivial
    $y\in G_{[H]}\cap \ker(\pi) = G_{[H]} \cap B$.  Since $H$ is
    infinite cyclic we get $|H\cap H^{y}| \neq \infty$ by
    Lemma~\ref{lem:aux4} which is equivalent to $[H] \neq [H^{y}]$,
    and this is a contradiction to the assumption that $y\in G_{[H]}$.
    Therefore $G_{[H]} = \Comm_{G}(H)$ must be infinite cyclic.
\end{proof}

\begin{proposition}
    \label{prop:main-aux}
    Let $G$ be an arbitrary group and let $\frakF$ and $\frakG$ be
    families of subgroups of $G$ such that
    \begin{equation*}
        \Ffin(G) \subset \frakF \subset \frakG \subset\Fvc(G).
    \end{equation*}
    Assume that the commensurator $\Comm_{G}(H)\in \frakG$ for any 
    $H\in \frakG\setminus\frakF$, then every $H\in \frakG\setminus 
    \frakF$ is contained in a unique maximal element $H_{\max}\in 
    \frakG$ and $N_{G}(H_{\max}) = H_{\max}$.
\end{proposition}

\begin{proof}
    Since $H$ is an infinite virtually cyclic subgroup of $G$ it 
    follows that $G_{[H]}= \Comm_{G}(H)$ and thus $G_{[H]}\in\frakG$ 
    by assumption.
    
    Trivially we have that $H\leq G_{[H]}$.  If $K\in \frakG$ with
    $H\leq K$, then $H\sim K$ since $H$ is not finite, and for any
    $k\in K$ we get $[H^{k}] = [K^{k}] = [K] = [H]$.  Therefore any
    $k\in K$ stabilises $[H]$.  This implies $K\leq
    G_{[H]}$ and thus $G_{[H]}$ is maximal and unique in
    $\frakG\setminus \frakF$, that is $H_{\max} = G_{[H]}$.
    
    Finally, $H_{\max}\leq N_{G}(H_{\max}) \leq \Comm_{G}(H_{\max}) = 
    G_{[H_{\max}]} = H_{\max}$ and hence $H_{\max} = 
    N_{G}(H_{\max})$.
\end{proof}

Together with Lemma~\ref{lem:aux2}, we get then the following result:

\begin{corollary}
    \label{cor:main-aux}
    Let $G = B\rtimes \Z$ and assume that $\Z$ acts freely by
    conjugation on the set of conjugacy classes of non-trivial
    elements of $B$.  Then every $H\in \Fvc(G)\setminus \Fvc(B)$ is
    contained in a unique maximal element $H_{\max}\in
    \Fvc(G)\setminus \Fvc(B)$ and $N_{G}(H_{\max}) = H_{\max}$.
    Moreover $\Fvc(B)\cap H =\{1\}$ for any $H\in \Fvc(G)\setminus
    \Fvc(B)$.
    \qed
\end{corollary}

%
%

\section{Juan-Pineda and Leary's Construction}
\label{sec:JPL}

Let $G$ be an arbitrary group and assume that $\frakF$ and $\frakG$ 
are two families of subgroups of $G$ which satisfy the conditions of 
Proposition~\ref{prop:main-aux}.

Let $H$ be a virtually cyclic group.  Then
by~\cite[p.~137]{juan-pineda-06}, there exists a unique maximal normal
finite subgroup $N$ of $H$ and precisely one of the cases occurs: $H$
is finite, $H/N$ is infinite cyclic (we say $H$ is \emph{orientable})
or $H/N$ is infinite dihedral (we say $H$ is \emph{non-orientable}).

The following result is a natural generalisation of Proposition~9 and
Corollary~10 in~\cite[pp.~140f.]{juan-pineda-06}.  Juan-Pineda and Leary's
proof can be used unchanged to verify these statements.

\begin{proposition}[Juan-Pineda and Leary]
    \label{prop:JPL}
    Assume that every $H\in \frakG\setminus \frakF$ is contained in a
    unique maximal element $H_{\max}\in \frakG$ and $N_{G}(H_{\max}) =
    H_{\max}$.  Moreover, assume that $\frakF\cap H\subset \Ffin(H)$
    for every $H\in \frakG\setminus \frakF$.  Let $\calC$ be a
    complete set of representatives of conjugacy classes of the
    maximal elements of $\frakG\setminus \frakF$.  Denote by
    $\calC_{o}$ and the set of orientable elements of $\calC$ and
    denote by $\calC_{n}$ the set of non-orientable elements of
    $\calC$.  Then a model for $E_{\frakG}G$ can be obtained from
    model for $E_{\frakF}G$ by attaching
    \begin{enumerate}
        \item  orbits of $0$-cells indexed by $\calC$;
    
        \item  orbits of $1$-cells indexed by $\calC_{o} \cup 
	\{1,2\}\times \calC_{n}$;
    
        \item  orbits of $2$-cells indexed by $\calC$.
    \end{enumerate}
    Furthermore, a model for $B_{\frakG}G$ can be obtained from a
    model for $B_{\frakF}G$ by attaching $2$-cells indexed by
    $\calC_{o}$.
    \qed
\end{proposition}

Juan-Pineda and Leary's construction is essentially the following.
For each maximal $H\in \frakG\setminus \frakF$ there exists a
$1$-dimensional model $E_{H}$ for $\underline{E}H$ which is
homeomorphic to the real line.  The model $E_{H}$ has one orbit of
$1$-cells and either one or two orbits of $0$-cells depending on if
$H$ is orientable or non-orientable.  Let $X$ be a model for
$E_{\frakF}G$.  Then a model for $E_{\frakG}G$ is obtained from $X$ by
attaching pieces of the form
\begin{equation*}
    X_{H} := E_{H} \times [0,1] \bigm/ \sim
\end{equation*}
to $X$ for each maximal element $H\in\frakG\setminus \frakF$ using
suitable equivariant maps $\eta\: E_{H}\times \{0\} \to X$ and where
the equivalence relation ``$\sim$'' identifies all elements $(x,1)\in
X_{H}$.  Juan-Pineda and Leary have shown how to implement this
construction such that we get a $G$-CW-complex with the desired
properties.

Note that in the case $\frakG= \Fvc(G)$ and $\frakF=\Ffin(G)$, we
recover the original statements in~\cite{juan-pineda-06}.  However, we
apply it to the case that $G = B\rtimes\Z$, $\frakF = \Fvc(B)$ and
$\frakG = \Fvc(G)$.  If $\Z$ acts freely by conjugation on the set of
conjugacy classes of non-trivial elements of $B$, then
Corollary~\ref{cor:main-aux} tells us that we can use
Proposition~\ref{prop:JPL} in order to construct a model for $\uu EG$
from a model for $E_{\Fvc(B)}G$.  However, in order to obtain this way
a nice model for $\uu EG$ we need to have a nice model for
$E_{\Fvc(B)}G$ to start with.  In the next section we will give a
general construction for such a model if a nice model for $\uu EB$ is
given.

%
%

\section{Constructing a Model for $E_{\frakF}G$ from a Model for
$E_{\frakF}B$}
\label{sec:model}

We carry out the construction in a setting that is more general than
in Section~\ref{sec:JPL}.  Let $G= B\rtimes \Z$ be an arbitrary
infinite cyclic extension, where $\Z$ acts on $B$ via an automorphism
$\varphi\in \Aut(B)$.  Let $\frakF$ be a family of subgroups of $B$.
We assume that $\frakF$ is invariant under the automorphism $\varphi$,
that is, $\varphi^{k}(H)\in \frakF$ for every $H\in \frakF$ and $k\in
\Z$.  This implies that $H\in \frakF$ if and only if
$\varphi(H)\in\frakF$ for any subgroup $H$ of~$B$.  Furthermore this
implies that $\frakF$ is not just a family of subgroups of~$B$ but
also a family of subgroups of~$G$.

We begin our construction with the assumption that we are given a
model $X$ for $E_{\frakF}B$.  For each $k\in \Z$ let $X_{k}$ be a
copy of $X$ seen as a set. We define a $B$-action
\begin{equation*}
    \Phi_{k}\: B\times X_{k} \to X_{k}
\end{equation*}
on $X_{k}$ by $\Phi_{k}(g,x) := \varphi^{-k}(g) x$.

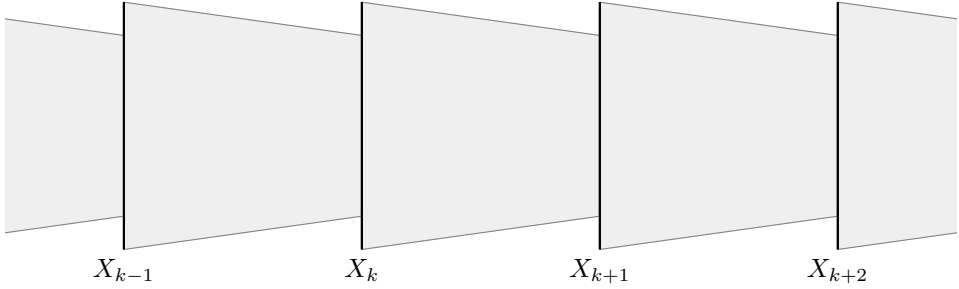
\begin{figure}[tbp]
    \begin{center}
        \begin{tikzpicture}[scale=\ScaleFigure]	
	    
        \clip  (-3.75,-1.66) rectangle (6.25,1.65);        
            
	\foreach \i in {-2,-1,0,1,2} {
	    
	    \coordinate (A1) at (2.5*\i    ,-1.3);
	    \coordinate (B1) at (2.5*\i    , 1.3);
	    \coordinate (A2) at (2.5*\i + 2.5,-1.3);
	    \coordinate (B2) at (2.5*\i + 2.5, 1.3);
	    \coordinate (a2) at (2.5*\i + 2.5,-0.95);
	    \coordinate (b2) at (2.5*\i + 2.5, 0.95);
	    
	    \filldraw [lightgray, nearly transparent]
	    (A1) -- (a2) -- (b2) -- (B1) -- cycle;
	
	    \draw [gray] (A1) -- (a2) (B1) -- (b2);
	    
	    \draw [thick] 
	    (B1) -- (A1)
	    (B2) -- (A2);
            }
	    
	    \draw
	    (-2.5,-1.3) node [below]{$X_{k-1}$}
	    (0,-1.3) node [below]{$X_{k}$}
	    (2.5,-1.3) node [below]{$X_{k+1}$}
	    (5,-1.3) node [below]{$X_{k+2}$};
	\end{tikzpicture}
    \end{center}
    \caption{A schematic picture of the $B$-CW-complex $Y$.}
    \label{fig:model}
\end{figure}

Since $X_{0}$ and $X_{1}$ are models for $\uu EB$ there exists a
$B$-map $f\: X_{0} \to X_{1}$.  In other words $f$ is a continuous map
$f\: X\to X$ which satisfies $f(gx) = \varphi^{-1}(g)f(x)$ for every
$x\in X$ and $g\in B$.  By the equivariant Cellular Approximation
Theorem~\cite[p.~32]{luck-89} we may assume without loss of
generality that
$f$ is cellular.  Denote by $X_\infty$ the disjoint union of
$B$-spaces
\begin{equation*}
    X_{\infty} := \coprod_{k\in \Z} (X_{k} \times [0,1])
\end{equation*}
and let $Y$ be the quotient space
\begin{equation*}
    Y := X_{\infty} / \sim
\end{equation*}
under the equivalence relation generated by $(x,1) \sim (f(x), 0)$  
whenever 
$x\in X_{k}$ and $f(x)\in X_{k+1}$ for some $k\in \Z$. Since $f$ is a 
cellular $B$-map it follows that $Y$ is a $B$-CW-complex. Essentially, 
it is a mapping telescope which extends to infinity in both 
directions, see Figure~\ref{fig:model}. Note that if $X$ is an $n$-dimensional
$B$-CW-complex, then $Y$ is $(n+1)$ dimensional $B$-CW-complex.

\begin{lemma}
    \label{lem:pre-model}
    The $B$-CW-complex $Y$ is a model for $E_{\frakF}B$.
\end{lemma}

\begin{proof}
    Let $H$ be a subgroup of $B$ such that $H\notin \frakF$ and let 
    $x\in X_{k}$ for some $k\in \Z$. Since $\frakF$ is assumed to be 
    invariant under the automorphism $\varphi$ we have 
    $\varphi^{-k}(H) \notin \frakF$. Therefore there exists a $h\in 
    H$ such that $\varphi^{-k}(h)x \neq x$. But then
    \begin{equation*}
        \Phi_{k}(h,x) = \varphi^{-k}(h)x \neq x,
    \end{equation*}
    which implies that $x\notin X_{k}^{H}$. It follows that 
    $X_{k}^{H} = \emptyset$ and therefore also $Y_{k}^{H}=\emptyset$ 
    for all $k\in \Z$. Hence $Y^{H}=\emptyset$.
    
    On the other hand, consider the case that $H\in \frakF$.  Since
    the family $\frakF$ is assumed to be invariant under the
    automorphism $\varphi$ it follows that $\varphi^{k}(H)\in \frakF$
    for every $k\in \Z$.  Then $X_{k}^{H} = X^{\varphi^{k}(H)}$ is
    contractible for every $k\in \Z$.  Choose for every $k\in \Z$ an
    element $x_{k}\in X_{k}^{H}$ and a contracting homotopy $h_{k}\:
    X_{k}^{H} \homotop \{x_{k}\}$.  This induces a contracting
    homotopy $h\: Y^{H} \homotop L$ where $L$ is a subspace of $Y^{H}$
    which is homeomorphic to the real line~$\R$.  Since $L$ is
    contractible it follows that $Y^{H}$ is contractible.
\end{proof}

For every $(x,t) \in X_{k}\times [0,1]$ and $(g,r)\in
G$ set 
\begin{equation*}
    \Psi(g,r) := (\Phi_{k+r}(g,x), t) \in
    X_{k+r} \times [0,1].
\end{equation*}
Straight forward calculation shows that this induces a well defined
action
\begin{equation*}
    \Psi\: G\times Y \to Y
\end{equation*}
of $G$ on $Y$, which extends the already existing $B$-action on
$Y$.  If $(g,r)\in G\setminus B$, then $r\neq 0$ and therefore clearly
$\Psi((g,r),x) \neq x$ for any $x\in Y$.  Then together with
Lemma~\ref{lem:pre-model} this implies that $Y$ is an
$(n+1)$-dimensional model for $E_{\frakF}G$.  Altogether we have then
shown the following result.

\begin{proposition}
    \label{prop:model}
    Let $G= B\rtimes \Z$ be an arbitrary infinite cyclic extension
    where $\Z$ acts on $B$ via an automorphism $\varphi\in \Aut(B)$.
    Let $\frakF$ be a family of subgroups of $B$ which is invariant
    under the automorphism $\varphi$.  If there exists an
    $n$-dimensional model for $E_{\frakF}B$ then there exists an
    $(n+1)$-dimensional model for~$E_{\frakF}G$.
    \qed
\end{proposition}

%
%

\section{Examples}
\label{sec:examples}

Strictly descending HNN-extensions are a natural source for candidates
for infinite cyclic extensions $G=B\rtimes \Z$ where $\Z$ acts freely
by conjugation on the set of conjugacy classes of the non-trivial
elements of $B$.

The general setup is the following.  Let $B_{0}$ be a group and
$\varphi\: B_{0}\to B_{0}$ a monomorphism which.
Recall that the descending HNN-extension determined by this data is
the group $G$ given by the presentation
\begin{equation*}
    G := \langle B_{0}, t \mid t^{-1}xt = \varphi(x)
    \text{ for all $x\in B_{0}$}\rangle
\end{equation*}
and this group is usually denoted by $B_{0} *_{\varphi}$ in the
literature.  The group $B_{0}$ is called the base group of the
HNN-extension.  The HNN-extension is called \emph{strictly descending}
if the monomorphism $\varphi$ is not an isomorphism.  We consider
$B_{0}$ in the obvious way as a subgroup of $G$.

Conjugation by $t\in G$ defines an automorphism of $G$ which agrees 
on $B_{0}$ with $\varphi$ which we will therefore denote by the same 
symbol. In other words, the monomorphism $\varphi\: B_{0} \to B_{0}$ 
extends to the whole group $G$ if we set
\begin{equation*}
    \varphi\: G\to G, x\mapsto \varphi(x) := t^{-1}xt.
\end{equation*}
For each $k\in \Z$ we set $B_{k} := \varphi^{k}(B_{0})$. We 
obtain this way a descending sequence
\begin{equation*}
    \ldots \supset B_{-2} \supset B_{-1} \supset B_{0}
    \supset B_{1} \supset B_{2} \supset \ldots
\end{equation*}
of subgroups of $G$.  This sequence of subgroups is strictly
descending if and only if the HNN-extension is strictly descending.
We denote the directed union of all these $B_{k}$ by~$B$.  The
automorphism $\varphi$ restricts to an automorphism of $B$ which is
therefore a normal subgroup of $G$.  It is standard fact that we can
write $G$ as the semidirect product $G = B\rtimes \Z$ where $\Z$ acts
on $B$ via the automorphism~$\varphi$ restricted to~$B$.

\begin{lemma}
    \label{lem:example}
    Assume that for every non-trivial $x\in B_{0}$ there exists a
    $k\in \N$ such that $x\notin \varphi^{k}(B_{0})$.  Given $x\in
    B_{0}$ denote by $[x]$ the set of all elements in $B_{0}$ which
    are conjugate in $B_{0}$ to $x$.  Assume that for each $x\in B$ we
    are given a finite subset $[x]' \subset [x]$ such that
    $\varphi([x]') \subset [\varphi(x)]' $ for every $x\in B_{0}$.
    Then $\Z$ acts freely on the set of conjugacy classes of
    non-trivial elements of $B$.
\end{lemma}

\begin{proof}
    We suppose that $\Z$ does not act freely on the set of conjugacy
    classes of non-trivial elements of $B$.  Then there exists $x\in
    B$ and $n\geq 1$ such that $\varphi^{n}(x)$ is conjugate in $B$ to
    $x$.  Without any loss of generality we may assume that $x\in
    B_{0}$ (otherwise replace $x$ by $\varphi^{k}(x)$ for a suitable
    $k\in \N$).  Furthermore, without any loss of generality we may
    assume that $x\in [x]'$.  Finally we may assume without any loss
    of generality that $\varphi^{n}(x)$ is actually conjugate in
    $B_{0}$ to $x$ (otherwise, again, replace $x$ by $\varphi^{k}(x)$
    for a suitable $k\in \N$).
    
    Now $\varphi^{rn}(x)\in [x]'$ for any $r\geq 1$.  Since $[x]'$ is
    finite this implies that $\varphi^{rn}(x) = \varphi^{sn}(x)$ for
    some $s>r$.  Therefore $\varphi^{(s-r)n}(x) = x$ and since
    $(s-r)n>0$ it follows that $x\in \varphi^{k}(B_{0})$ for any $k\in
    \N$.  However, this is a contradiction on the hypothesis that $\Z$
    does not act freely on the set of conjugacy classes of non-trivial
    elements of~$B$.  Therefore the opposite must be true.
\end{proof}

Note that the requirement that for every non-trivial element $x\in
B_{0}$ there exists a $k\in \N$ such that $x\notin \varphi^{k}(B_{0})$
implies that the descending HNN-extension $G = B_{0} *_{\varphi}$ is
actually strictly descending.  Furthermore, we can conclude from it
that the intersection of the groups $B_{k}$, $k\in \Z$, is trivial.

\begin{example}
    \label{ex:abelian_base_group}
    Let $B_{0}$ be an abelian group and $\varphi\: B_{0}\to B_{0}$ a
    monomorphism such that for every non-trivial $x\in B_{0}$ there
    exists a $k\in \N$ such that $x\notin \varphi^{k}(B_{0})$.  Since
    $B_{0}$ is abelian, each conjugacy class $[x]$ of elements in
    $B_{0}$ contains precisely one element.  Therefore
    Lemma~\ref{lem:example} states that $\Z$ acts freely by
    conjugation on the set of non-trivial elements of $B$.  In
    particular we can use Proposition~\ref{prop:JPL} to obtain a model
    for $\uu EG$ from a model for $E_{\Fvc(B)}G$.
\end{example}

Let $B_{0}$ be a free group.  An element $x\in B_{0}$ is called
\emph{cyclically reduced} if it cannot be written as $x = u^{-1}yu$
for some non-trivial $u, y\in B_{0}$. It follows 
from~\cite[pp.~33ff.]{magnus-76} that every element $x\in B_{0}$ is 
conjugate to a cyclically reduced element $x'$ and that there are 
only finitely many cyclically reduced elements in $B_{0}$ which are 
conjugate to $x$. Therefore
\begin{equation*}
    [x]' := \{ x' \in [x] : \text{$x'$ is cyclically reduced}\}
\end{equation*}
is a finite subset of $[x]$ for every $x\in B_{0}$.  Then the
following two assumptions on the monomorphism $\varphi\: B_{0}\to
B_{0}$ are necessary in order to apply Lemma~\ref{lem:example}:
\begin{enumerate}
    \item \label{item:assumption1}
    For every non-trivial $x\in B_{0}$ there exists a $k\in \N$ such
    that $x\notin \varphi^{k}(B_{0})$;

    \item \label{item:assumption2}
    If $x$ is a cyclically reduced element in $B_{0}$, then so 
    $\varphi(x)$.
\end{enumerate}

\begin{example}
    \label{ex:free_base_group}
    Let $X$ be an arbitrary non-empty set and let $B_{0} := F(X)$ be
    the free group on the basis $X$.  Let $\{\alpha_{x}\}_{x\in X}$ be
    a collection of integers such that $|\alpha_{x}|\geq 2$ for every
    $x\in X$.  Consider the endomorphism $\varphi\: B_{0}\to B_{0}$
    that maps any basis element $x$ to $x^{\alpha_{x}}$.  Then
    $\varphi$ is a monomorphism which satisfies the
    assumptions~\eqref{item:assumption1} and~\eqref{item:assumption2}
    above.  Lemma~\ref{lem:example} tells us then that we can use
    Proposition~\ref{prop:JPL} to construct a model for $\uu EG$ from
    a model for $E_{\Fvc(B)}G$.
\end{example}

\begin{example}
    Another example of a strictly descending HNN-extension (in
    disguise) is the standard wreath product $A\wr \Z$ of an arbitrary
    group $A$ by $\Z$ which is defined as follows.  Let $A_{k}$ be a
    copy of $A$ for each $k\in \Z$.  Let $B$ be the coproduct of all
    these $A_{k}$ and let $\Z$ act on $B$ via $\varphi$ which maps
    $A_{k}$ identically onto $A_{k+1}$ for all $k\in \Z$.  Then
    \begin{equation*}
        A\wr \Z := B\rtimes \Z.
    \end{equation*}
    Since each $A_{k}$ is normal in $B$ the above definition of $\varphi$
    forces the action of $\Z$ on the set of conjugacy classes of
    non-trivial elements of $B$ to be free.  Therefore we can apply
    Proposition~\ref{prop:JPL} in this case, too.
\end{example}

%
%

\section{Dimensions}

Given a family $\frakF$ of subgroups of $G$, a model for $E_{\frakF}G$
is only defined uniquely up to $G$-homotopy.  Consider a model for
$E_{\frakF}G$.  One particular invariant of the group $G$ is called
the \emph{geometric dimension} of $G$ with respect to the family
$\frakF$, and this is defined as being the least possible dimension of
a model for $E_{\frakF}G$.  It is denoted by $\gd_{\frakF}G$ and may
be infinite.  In the case that $\frakF =\{1\}$ we recover the
classical geometric dimension of the group $G$.  In the case that
$\frakF = \Fvc(G)$ we denote the geometric dimension by $\uugd G$.

\begin{theorem}
    \label{thrm:dimension}
    Let $G = B\rtimes \Z$ and assume that $\Z$ acts freely via 
    conjugation on the conjugacy classes of non-trivial elements of 
    $B$. Then $B$ is not virtually cyclic and 
    \begin{equation*}
	\uugd B\leq \uugd G \leq \uugd B + 1.
    \end{equation*}
\end{theorem}

\begin{proof}
    Since (in general) a model for $\uu EG$ is always a model for $\uu
    EB$ via restriction, we have that the second inequality is the
    only non-trivial one.  If $X$ is an $n$-dimensional model for $\uu
    EB$, then the joint construction in Section~\ref{sec:model} gives
    an $(n+1)$-dimensional model for $E_{\Fvc(B)}G$.
    
    Since $B$ cannot be virtually cyclic we have $n+1 \geq 2$ and
    attaching cells of dimension at most $2$ does not increase the
    dimension of the resulting space.  Hence
    Proposition~\ref{prop:JPL} yields an $(n+1)$-dimensional model for
    $\uu EG$ and this concludes the proof.  
\end{proof}

\begin{corollary}
    \label{cor:dimension}
    Let $G = B_{0} *_{\varphi}$ is a descending HNN-extension as in 
    Section~\ref{sec:examples}. If $G = B\rtimes \Z$ satisfies the 
    conditions of the previous proposition then
    \begin{equation*}
	\uugd B_{0} \leq \uugd G \leq  \uugd B_{0} + 2.
    \end{equation*}
\end{corollary}

\begin{proof}
    As exploited previously, since $B_{0}$ is a subgroup of $G$, the
    second inequality is the only non-trivial part of the statement.
    The group $B$ is the countable direct union of the conjugates of
    $B_{0}$ in $G$.  Therefore an $n$-dimensional model for $\uu
    EB_{0}$ gives rise to an $(n+1)$-dimensional model for $\uu
    EB$ by a construction of Lück and
    Weiermann~\cite[pp.~11ff.]{luck-07}.  Now the claim follows from
    the previous proposition.
\end{proof}

\begin{example}
    Let $G = B_{0} *_{\varphi}$ be a descending HNN-extension with
    $B_{0}$ a free group.  If $B_{0}$ has rank $1$, then $G$ is a
    soluble Baumslag--Solitar group and this case is treated below in
    Theorem~\ref{thrm:BSgroups}.  Thus we may assume that $B_{0}$
    has rank at least $2$.  Free groups are torsion free and act
    freely on a tree which is therefore a $1$-dimensional model for
    $\underline{E} B_{0}$.  Free groups are word hyperbolic and
    therefore Proposition~9 in~\cite{juan-pineda-06} states the
    existence of a $2$-dimensional model for $\uu E B_{0}$.  On the
    other hand by Remark~16 in~\cite{juan-pineda-06} there cannot
    exist a model for $\uu EB_{0}$ less than $2$.  Therefore $\uugd
    B_{0} =2$.  Now the direct union $B$ of all conjugates of $B_{0}$
    in $G$ is locally free and therefore does not contain a subgroup
    of isomorphic to~$\Z^{2}$.  Then Lemma~\ref{lem:aux1} states that
    we can apply Corollary~\ref{cor:dimension} if and only if $G$ does
    not contain a subgroup isomorphic to~$\Z^{2}$.  Therefore we get
    in this case the estimation $2 \leq \uugd G \leq 4$.
\end{example}

\begin{example}
    Consider the wreath product $G = A\wr \Z$ where $A$ is a countable
    locally finite group.  Then
    \begin{equation*}
	B := \coprod_{k\in \Z} A
    \end{equation*}
    is a also a countable locally finite group.  Then $B$ is a
    countable colimit of its finite subgroups $B_{\lambda}$ and
    Theorem~4.3 in~\cite{luck-07} gives the estimate $\uugd B \leq
    \sup \{ \uugd B_{\lambda}\} + 1$.  Since $\uugd B_{\lambda} = 0$
    for every $\lambda$ and we get $\uugd B\leq 1$.  On the other hand
    $B$ is not virtually cyclic and therefore $\uugd B\neq 0$, that is
    $\uugd B=1$.  We have seen that $G$ does satisfy the requirements
    of Theorem~\ref{thrm:dimension}.  Therefore we get the estimate $1
    \leq \uugd (A\wr \Z)\leq 2$.  Note that $A\wr \Z$ is not locally
    virtually cyclic.  Therefore we can furthermore exclude the
    possibility $\uugd (A\wr \Z) =1$ using the next proposition.  Thus
    we have altogether
    \begin{equation*}
        \uugd (A\wr \Z) = 2.
    \end{equation*}
    Note that the smallest concrete example of a group of this type is
    the lamplighter group $L = \Z_{2} \wr \Z$ where $\Z_{2}$ is the
    cyclic group of the integers modulo~$2$.
\end{example}

\begin{proposition}
    \label{prop:gdG=1}
    Let $G$ be a group with $\uugd G=1$.  Then $G$ is a locally
    virtually cyclic and not finitely generated.
\end{proposition}

\begin{proof}
    The assumption $\uugd G=1$ implies that $G$ has a tree $T$ as a
    model for $\uu EG$.  If $G$ is finitely generated then Corollary~3
    to Proposition~25 in~\cite[pp.~64]{serre-03} states that
    $T^{G}\neq \emptyset$ which implies that $G$ is virtually cyclic.
    But this means $\uugd G=0$ which contradicts the assumption that
    $\uugd G=1$.  Therefore $G$ is not finitely generated.
    
    Let $H$ be a finitely generated subgroup of $G$.  Since $H$ is a
    subgroup of $G$ we have $\uugd H \leq \uugd G$.  However, since
    $H$ is finitely generated the the case $\uugd H\neq 1$ cannot
    occur.  Therefore $\uugd H=0$ which implies that $H$ is virtually
    cyclic.  Hence $G$ is locally virtually cyclic.  
\end{proof}

We conclude this article with a complete answer to the geometric
dimension of the soluble Baumslag--Solitar groups with respect to the
family of virtually cyclic subgroups.  These groups belong to a class
of two-generator and one-relator groups introduced by Baumslag and
Solitar in~\cite{baumslag-62}. Their class consists of the groups 
\begin{equation*}
    BS(p,q) = \langle x,t \mid t^{-1}x^{p}t  = x^{q}\rangle.
\end{equation*}
where $p$ and $q$ are non-zero integers.  The soluble
Baumslag--Solitar groups are precisely those groups which are 
isomorphic to $BS(1,m)$ for some $m\neq 
0$. These groups can also be written as
\begin{equation*}
    BS(1,m) = \Z[1/m]\rtimes \Z,
\end{equation*}
where $\Z[1/m]$ is the subgroup of the rational numbers $\Q$ generated
by all powers of $1/m$ and where $\Z$ acts on $\Z[1/m]$ by
multiplication by $m$.  The group $BS(1,1)$ is $\Z^{2}$ and
$BS(1,-1)$ is the Klein bottle group $\Z\rtimes \Z$.  If $|m|\geq 2$,
then $BS(1,m)$ belongs to the case described in
Example~\ref{ex:abelian_base_group}, as well as to the case described
in Example~\ref{ex:free_base_group}.

\begin{theorem}
    \label{thrm:BSgroups}
    Let $G = BS(1,m)$ be a soluble Baumslag--Solitar group. Then
    \begin{equation*}
        \uugd G =
	\begin{cases}
	    3 & \text{if $|m|=1$,}
	    \\[0.3ex]
	    2 & \text{otherwise.}
	\end{cases}
    \end{equation*}
\end{theorem}

\begin{proof}
    The case $|m|=1$ has been answered in Example~3 and Remark~15
    in~\cite{juan-pineda-06}. 
    
    Thus we assume that $|m|\geq 2$.  The group $\Z[1/m]$ is countable
    and locally virtually cyclic.  Therefore $\uugd \Z[1/m]\leq 1$
    by~\cite[pp.~11ff.]{luck-07}.  The action of $\Z$ on the conjugacy
    classes of non-trivial elements of $\Z[1/m]$ is free and thus we
    can apply Theorem~\ref{thrm:dimension} to obtain the estimate $\uugd
    G\leq 2$.  However, $\uugd G$ cannot be zero as $G$ is not
    virtually cyclic.  Moreover $G$ is finitely generated and therefore
    $\uugd G\neq 1$ by Proposition~\ref{prop:gdG=1}.  Thus $\uugd
    G = 2$ is the only remaining possibility.
\end{proof}    

We conclude this article with a  model for $\uu EG$ for the
soluble Baumslag--Solitar groups $G=BS(1,m)$, $|m|\neq 1$, 
which is nicer than the one obtained from the general
construction in Section~\ref{sec:model} together 
with Proposition~\ref{prop:JPL}.  The group
$G$ can be realised as the fundamental group of a graph $(G,Y)$ of
groups in the sense of~\cite{serre-03} where~$Y$ is a loop and where
the vertex and edge groups are all infinite cyclic.  Let~$T$ be the
Bass--Serre tree associated with this graph of groups.  It follows
that~$T$ is not only a model for $\uu E\Z[1/m]$, but also a model for
$E_{\Fvc(\Z[1/m])} G$.  We can apply Proposition~\ref{prop:JPL} and
obtain a model $X$ for $\uu EG$ by attaching orbits of $0$-, $1$- and
$2$-cells to~$T$ indexed by the infinitely many conjugacy classes of
maximal infinite cyclic subgroups of $G$ which are not contained in
the subgroup $\Z[1/m]$.  Furthermore, since $Y = T/G$, we obtain a
model for $\uu BG$ by attaching infinitely many $2$-cells along the
loop $Y$.  That is, a model for $\uu BG$ is given by the quotient
space $(D^{2} \times \Z) / \!\sim$ where the equivalence relation is
given by $(x,k)\sim (x,l)$ for all $x\in \partial D^{2} = S^{1}$ and
all $k,l\in\Z$.

\bibliographystyle{alpha}
\bibliography{math}

\end{document}